\documentclass[12pt]{article}%
\usepackage{amsmath}
\usepackage{amsfonts}%
\usepackage{amssymb}%
\usepackage{graphicx}
\providecommand{\U}[1]{\protect\rule{.1in}{.1in}}
\newtheorem{theorem}{Theorem}

\newtheorem{corollary}[theorem]{Corollary}

\newtheorem{lemma}[theorem]{Lemma}

\newtheorem{proposition}[theorem]{Proposition}

\newenvironment{proof}[1][Proof]{\noindent\textbf{#1.} }{\ \rule{0.5em}{0.5em}}
\begin{document}

\title{Some implications of a conjecture of Zabrocki to the action of $S_{n}$ on
polynomial differential forms }
\author{Nolan R. Wallach}
\maketitle

\begin{abstract}
The symmetric group acts on polynomial differential forms on $\mathbb{R}^{n}$
through its action by permuting the coordinates. In this paper the $S_{n}%
$-invariants are shown to be freely generated by the elementary symmetric
polynomials and their exterior derivatives. A basis of the alternants in the
quotient of the ideal generated by the homogeneous invariants of positive
degree is given. In addition, the highest bigraded degrees are given for the
quotient. All of these results are consistant with predictions derived by
Garsia and Romero from a recent conjecture of Zabrocki.

\end{abstract}

\section{Introduction}

Let $A$ be the free algebra generated by  $x_{1},..,x_{n},y_{1},...,y_{n}%
,\theta_{1},...,\theta_{n}$ subject to the following relations: the
$\theta_{i}$ commute with the $x_{j}$and $y_{k}$, the $x_{i}$ and $y_{j}$ all
commute and $\theta_{i}\theta_{j}=-\theta_{j}\theta_{i}$.  Then $A$ is triply
graded by degree. Let $S_{n}$ act on $A$ by permuting the indices in the
$x_{i},y_{i}$ and $\theta_{i}$ in the same way. Let $I$ be the ideal in $A$
generated by the homogeneous $S_{n}$--invariants of positive degree. Mike
Zabrocki has made a conjecture about the trigraded Hilbert series of $A/I$ in
[Z] in relation to the Delta conjecture of algebraic combinatorics. Adriano
Garsia showed me what Zabrocki's conjecture implies for only one set of
commuting variables. This is equivalent to the action of $S_{n}$ on polynomial
differential forms and he suggested a consequence: the Hilbert series of the
alternants in the corresponding specialized quotient (see Theorem 13). Marino
Romero derived an upper bound on the possible degrees of the commuting
variables when the degree in the alternating variable is fixed for an element
of the quotient (see Theorem 14). The purpose of this paper is to prove those consequences.

In addition we give a suggested way of generating the $S_{n}$--harmonic
polynomial differential forms from the $S_{n}$--invariant polynomials as a
conjecture at the end of the paper.

Many of the techniques of this paper can be applied in greater generality,
say, to general finite groups. Results analogous to those in [W] for the Weyl
algebra are true for the \textquotedblleft super-Weyl\textquotedblright%
\ algebra and could be useful in the proof of the conjecture (either mine or
Zabrocki's). We thank Brendon Rhodes for pointing out the paper of Louis
Solomon, [S].

\section{The setting}

Let $\Omega_{n}$ be the free algebra algebra over $\mathbb{R}$ on
$x_{1},...,x_{n},y_{1},...,y_{n}$ subject to the relations $x_{i}x_{j}%
=x_{j}x_{i,}y_{i}x_{j}=x_{j}y_{i},y_{i}y_{j}=-y_{j}y_{i}$. Then $\Omega_{n}$
is isomorphic with the algebra of differential forms on $\mathbb{R}^{n}$ with
polynomial coefficients if we take $y_{i}=dx_{i}$. Let $S_{n}$ act diagonally
by permuting the indices of the $x$'s and the $y$'s in the same way . This is
the same as the action of $S_{n}$ on the differential forms. It will be
convenient to use the language of differential forms. Our first task is to
describe the invariants. We can also look at $\Omega_{n}$ as space
$\mathbb{R}[x_{1},...,x_{n}]\otimes\wedge\mathbb{R}^{n}$ with $\wedge
\mathbb{R}^{n}$ the Grassmann algebra on the vector space $\mathbb{R}^{n}$ and
the tensor factors commuting. Then in particular $\Omega^{n}$ is a module for
the $S_{n}$--invariants in $\mathbb{R}[x_{1},...,x_{n}]$, that is,
$\mathbb{R}[x_{1},...,x_{n}]^{S_{n}}$. Let $H$ be the space of all $S_{n}
$--harmonic polynomials. Then, as is well known, the map%
\[
\mathbb{R}[x_{1},...,x_{n}]^{S_{n}}\otimes H\rightarrow\mathbb{R}%
[x_{1},...,x_{n}]
\]
given by $f\otimes h\mapsto fh$ is a linear bijection. This implies that the
linear map%
\[
\mathbb{R}[x_{1},...,x_{n}]^{S_{n}}\otimes\left(  H\otimes\wedge\mathbb{R}%
^{n}\right)  \rightarrow\Omega_{n}%
\]
given by $f\otimes\omega\mapsto f\omega$ is a linear bijection. We will denote
the multiplication in $\wedge\mathbb{R}^{n}$ by $\alpha,\beta\rightarrow
\alpha\wedge\beta$. Note that the space of $S_{n}$--invariants in
$\wedge\mathbb{R}^{n}$, $\left(  \wedge\mathbb{R}^{n}\right)  ^{S_{n}}$ is two
dimensional with homogeneous basis%
\[
1,u=\sum_{i=1}^{n}y_{i}=\sum_{i=1}^{n}dx_{i}.
\]
Let $\wedge^{k}\mathbb{R}^{n}$ be the space spanned by products of exactly $k$
of the $y_{i}$. Note that $\mathbb{R}^{n}=F\oplus\mathbb{R}u$ with $F$ the
irreducible representation corresponding to the partition $\left(
n-1,1\right)  $ (the first hook). Also $\wedge^{k}\mathbb{R}^{n}=\wedge
^{k}F\oplus u\wedge\wedge^{k-1}F.$ As is well known, $\wedge^{k}F$ is the
irreducible representation of $S_{n}$ corresponding to the partition $\left(
n-k,1^{k}\right)  $ (the $k$--th hook).

If $\Lambda$ is a partition of $n$ let $F^{\Lambda}$ denote the corresponding
irreducible representation of $S_{n}.$ Thus $F=F^{(n-1,1)}$.

Put the inner product $(...,...)$ on the real valued functions on $S_{n}$
given by
\[
(f,g)=\frac{1}{n!}\sum_{x\in S_{n}}f(x)g(x).
\]
The Schur orthogonality relations say that if $\chi_{\Lambda}$ is the
character of $F^{\Lambda}$ then
\[
(\chi_{\Lambda},\chi_{\mu})=\delta_{\Lambda,\mu}.
\]

Let $H=\oplus_{l=0}^{\binom{n}{2}}H^{k}$ with $H^{k}$ the homogenous elements
of $H$ of degree $k$. Let $m_{j,\Lambda}$ denote the multiplicity of the
representation $F^{\Lambda}$ in $H^{k}$. That is, if $\eta_{j}$ is the
character of $H^{j}$ as a representation of $S_{n}$ then%
\[
m_{j,\Lambda}=(\eta_{j},\chi_{\Lambda}).
\]
Note that the multiplicity of $F^{n-k,1^{k}}$ in $\wedge^{l}\mathbb{R}^{n}$ is
$0$ if $l\notin\left\{  k,k+1\right\}  $ and one if $l\in\left\{
k,k+1\right\}  $. Set $\gamma_{k}$ equal to the character of the
representation $\wedge^{k}\mathbb{R}^{n}$. Then $\gamma_{k}=\chi_{\left(
n-k,1^{k}\right)  }+\chi_{(n-k+1,1^{k-1})}$ Define the graded character of $H
$ to be $\sum_{j=0}^{\binom{n}{2}}q^{j}\eta_{j}$ and that of $\wedge
\mathbb{R}^{n}$ to be $\sum_{k=0}^{n}t^{k}\gamma_{k}$ \ The bigraded character
of $H\otimes\wedge\mathbb{R}^{n}$ is $\sum_{j,k}q^{j}t^{k}\eta_{j}\gamma_{k}$.

\begin{lemma}
The bigraded character of $H\otimes\wedge\mathbb{R}^{n}$ is%
\[
\sum_{k,l}m_{l,\Lambda}q^{l}(t^{k}+t^{k+1})\chi_{\Lambda}\chi_{\left(
n-k,1^{k}\right)  }.
\]

\end{lemma}

\begin{proof}
By the above $\sum_{k=0}^{n}t^{k}\gamma_{k}=\sum_{k=0}^{n-1}(t^{k}%
+t^{k+1})\chi_{n-k,1^{k}}$.
\end{proof}

Noting that
\[
\left(  \chi_{\Lambda}\chi_{\mu},1\right)  =\left(  \chi_{\Lambda},\chi_{\mu
}\right)  =\delta_{\Lambda,\mu}%
\]
we have the corollary.

\begin{corollary}
\label{full-hilb}The bigraded Hilbert series of the $S_{n}$ invariants in
$H\otimes\wedge\mathbb{R}^{n}$ is%
\[
\sum_{k,l}m_{l,\left(  n-k,1^{k}\right)  }q^{l}(t^{k}+t^{k+1}).
\]
Thus the bigraded Hilbert series of $\Omega_{n}^{S_{n}}$ is
\[
\frac{\sum_{k,l}m_{l,\left(  n-k,1^{k}\right)  }q^{l}(t^{k}+t^{k+1})}%
{\prod_{j=1}^{n}\left(  1-q^{j}\right)  }.
\]

\end{corollary}

\section{The $S_{n}$ invariants in $\Omega_{n}.$}

Observe that if $f$ is an $S_{n}$--invariant polynomial in $x_{1},...,x_{n}$
then%
\[
df=\sum\frac{\partial f}{\partial x_{i}}dx_{i}%
\]
is an $S_{n}$ invariant in $\Omega_{n}$. Set (as usual) $p_{j}=p_{j,n}%
=\sum_{i=1}^{n}x_{i}^{j}$. Then, clearly $dp_{i_{1}}\wedge dp_{i_{2}}%
\wedge\cdots\wedge dp_{i_{k}}\in\Omega_{n}^{S_{n}}$.

The following result is a direct consequence of the theorem of [S] which is
the same statement for any finite reflection group. We include a proof for the
benefit of the audience only interested in the symmetric group and since its
corollary on the multiplicity of the hook representations is important to the
rest of the paper.

\begin{proposition}
$\Omega_{n}^{S_{n}}$ is the direct sum
\[
\oplus_{k=0}^{n}\oplus_{1\leq i_{1}<...<i_{k}\leq n}\mathbb{R}[x_{1}%
,...,x_{n}]^{S_{n}}dp_{i_{1}}\wedge dp_{i_{2}}\wedge\cdots\wedge dp_{i_{k}}.
\]

\end{proposition}

\begin{proof}
Set
\[
\Omega_{n}^{k}=\oplus_{1\leq i_{1}<...<i_{k}\leq n}\mathbb{R}[x_{1}%
,...,x_{n}]^{S_{n}}dx_{i_{1}}\wedge dx_{i_{2}}\wedge\cdots\wedge dx_{i_{k}}.
\]
We now prove the result induction on $n$ leaving the cases $n=1,2$ to the
reader. Assume the result for $n-1\geq2$. Let $\omega\in\left(  \Omega_{n}%
^{k}\right)  ^{S_{n}}$. Then
\begin{align*}
\omega & =\sum_{1\leq i_{1}<...<i_{k}<n}a_{i_{1}i_{2}...i_{k}}(x)dx_{i_{1}%
}\wedge dx_{i_{2}}\wedge\cdots\wedge dx_{i_{k}}+\\
& \sum_{1\leq l_{1}<...<l_{k-1}<n}b_{l_{1}l_{2}...l_{k-1}}(x)dx_{l_{1}}\wedge
dx_{l_{2}}\wedge\cdots\wedge dx_{l_{k-1}}\wedge dx_{n}.
\end{align*}
Set $x=(x^{\prime},x_{n})$ then writing $a_{i_{1}i_{2}...i_{k}}(x)=\sum
_{j}a_{i_{1}i_{2}...i_{k};j}(x^{\prime})x_{n}^{j}$ we have
\begin{align*}
\omega & =\sum_{1\leq i_{1}<...<i_{k}<n}\sum_{j}a_{i_{1}i_{2}...i_{k}%
;,j}(x^{\prime})x_{n}^{j}dx_{i_{1}}\wedge dx_{i_{2}}\wedge\cdots\wedge
dx_{i_{k}}+\\
& \sum_{1\leq l_{1}<...<l_{k-1}<n}b_{l_{1}l_{2}...l_{k};.j}(x^{\prime}%
)x_{n}^{j}dx_{l_{1}}\wedge dx_{l_{2}}\wedge\cdots\wedge dx_{l_{k-1}}\wedge
dx_{n}.
\end{align*}
Noting that
\[
\sum_{1\leq i_{1}<...<i_{k}<n}\sum_{j}a_{i_{1}i_{2}...i_{k};,j}(x^{\prime
})dx_{i_{1}}\wedge dx_{i_{2}}\wedge\cdots\wedge dx_{i_{k}}\in\Omega
^{k}(\mathbb{R}^{n-1})^{S_{n-1}}%
\]
and%
\[
\sum_{1\leq l_{1}<...<l_{k-1}<n}b_{l_{1}l_{2}...l_{k-1};.j}(x^{\prime
})dx_{l_{1}}\wedge dx_{l_{2}}\wedge\cdots\wedge dx_{l_{k-1}}\in\Omega
^{k-1}(\mathbb{R}^{n-1})^{S_{n-1}}.
\]
We can apply the inductive hypothesis and the fact that $\mathbb{R}%
[x_{1},...,x_{n-1}]^{S_{n-1}}=\mathbb{R[}p_{1,n-1},...,p_{n-1,n-1}]$ to see
that%
\[
\sum_{1\leq i_{1}<...<i_{k}<n}a_{i_{1}i_{2}...i_{k};,j}(x^{\prime})dx_{i_{1}%
}\wedge dx_{i_{2}}\wedge\cdots\wedge dx_{i_{k}}%
\]%
\[
=\sum\varphi_{i_{1}i_{2}...i_{k};,j}(p_{1,n-1},...,p_{n-1,n-1})dp_{i_{1}%
,n-1}\wedge dp_{i_{2},n-1}\wedge\cdots\wedge dp_{i_{k},n-1}%
\]
and%
\[
\sum_{1\leq l_{1}<...<l_{k-1}<n}b_{l_{1}l_{2}...l_{k-1};.j}(x^{\prime
})dx_{i_{1}}\wedge dx_{i_{2}}\wedge\cdots\wedge dx_{i_{k-1}}%
\]%
\[
=\sum_{1\leq l_{1}<...<l_{k-1}<n}\psi_{l_{1}l_{2}...l_{k-1};.j}(p_{1,n-1}%
,...,p_{n-1,n-1})dp_{l_{1},n-1}\wedge dp_{l_{2},n-1}\wedge\cdots\wedge
dp_{l_{k-1},n-1}.
\]
Observing that $p_{j,n-1}=p_{j,n}-x_{n}^{j}$ \ and thus $dp_{j,n-1}%
=dp_{j,n}-jx_{n}^{j-1}dx_{n}$. We can expand out once again and find
($p_{j}=p_{j,n}$)%
\[
\omega=\sum_{j,1\leq i_{1}<...<i_{k}<n}\alpha_{i_{1}i_{2}...i_{k};,j}%
(p_{1},..,p_{n-1})x_{n}^{j}dp_{i_{1}}\wedge dp_{i_{2}}\wedge\cdots\wedge
dp_{i_{k}}%
\]%
\[
+\sum_{j,1\leq l_{1}<...<l_{k-1}<n}\beta_{l_{1}l_{2}...l_{k-1};.j}%
(p_{1},..,p_{n-1})x_{n}^{j}dp_{l_{1}}\wedge dp_{l_{2}}\wedge\cdots\wedge
dp_{l_{k-1}}\wedge dx_{n}.
\]
The $S_{n}$ invariance implies that
\[
\omega=\frac{1}{n}\sum_{j,1\leq i_{1}<...<i_{k}<n}\alpha_{i_{1}i_{2}%
...i_{k};,j}(p_{1},..,p_{n-1})p_{j}dp_{i_{1}}\wedge dp_{i_{2}}\wedge
\cdots\wedge dp_{i_{k}}%
\]%
\[
+\frac{1}{n}\sum_{j,1\leq l_{1}<...<l_{k}<n}\frac{1}{j+1}\beta_{l_{1}%
l_{2}...l_{k-1};.j}(p_{1},..,p_{n-1})dp_{l_{1}}\wedge dp_{l_{2}}\wedge
\cdots\wedge dp_{l_{k-1}}\wedge dp_{j}.
\]
Finally,%
\[
p_{j}=u_{j}(p_{1},...,p_{n})
\]
with $u_{j}(t_{1},...,t_{n})$ a polynomial. So
\[
dp_{j}=\sum\frac{\partial u_{j}}{\partial t_{k}}(p_{1},...,p_{n})dp_{k}.
\]
This shows that
\[
\left(  \Omega_{n}^{k}\right)  \subset\sum_{I=1\leq i_{1}<...<i_{k}\leq
n}\mathbb{R}[x_{1},...,x_{n}]^{S_{n}}dp_{i_{1}}\wedge dp_{i_{2}}\wedge
\cdots\wedge dp_{i_{k}}..
\]
To show that the sum is direct observe that
\[
dp_{1}\wedge\cdots\wedge dp_{n}=n!\prod_{i<j}(x_{i}-x_{j})dx_{1}\wedge
\cdots\wedge dx_{n}.
\]
This implies that if
\[
\sum_{1\leq i_{1}<...<i_{k}\leq n}\gamma_{i_{1}i_{2}...i_{k}}(x)dp_{i_{1}%
}\wedge dp_{i_{2}}\wedge\cdots\wedge dp_{i_{k}}=0
\]
then $\gamma_{I}(x)=0$ if $\prod_{i<j}(x_{i}-x_{j})\neq0$. Thus the sum is direct.
\end{proof}

\begin{corollary}
$\Omega_{n}^{S_{n}}$ is freely generated by $p_{1},...,p_{n},dp_{1}%
,...,dp_{n}$ or $e_{1},...,e_{n},de_{1},...,de_{n}$(the $e_{i}$ are the
elementary symmetric functions).
\end{corollary}

The assertion using the elementary symmetric polynomials is probably true over
$\mathbb{Z}$.

\begin{corollary}
Set $\mu_{l,k}$ equal to the dimension of the space of $S_{n}$--invariant $k $
forms with polynomial coefficients homogeneous of degree $l$ in $x_{1}%
,...,x_{n}$. Then%
\[
\sum_{l,k}q^{l}t^{k}\mu_{l,k}=\prod_{j=1}^{n}\frac{1+q^{j-1}t}{1-q^{j}}.
\]

\end{corollary}

\begin{proof}
The Hilbert series of the algebra generated by $dp_{1},...,dp_{n}$ is
\[%
%TCIMACRO{\dprod _{j=1}^{n}}%
%BeginExpansion
{\displaystyle\prod_{j=1}^{n}}
%EndExpansion
\left(  1+q^{j-1}t\right)
\]
that for the algebra generated by $p_{1},...,p_{n}$ is
\[
\prod_{j=1}^{n}\frac{1}{1-q^{j}}.
\]

\end{proof}

\begin{corollary}
\label{mult}Recall that $m_{j,\Lambda}$ is the mutiplicity of $F^{\Lambda}$ in
$H^{j}$ then
\[
\sum m_{j,(n-k,1^{k})}q^{j}=e_{k}(q,q^{2},...,q^{n-1}).
\]

\end{corollary}

\begin{proof}
The lemma above and lemma \ref{full-hilb} imply that
\[
\overset{}{(\ast)}\sum_{k,j}m_{j,(n-k,1^{k})}q^{j}(t^{k}+t^{k+1})=\prod
_{j=1}^{n}\left(  1+q^{j-1}t\right)  .
\]
We prove the desired formula by induction on $k$. For the sake of simplicity
we will denote $m_{j,(n-k,1^{k})}$ by $m_{j,k}$ If $k=0$ then $m_{j,0}%
=\delta_{0,j}$ So we are looking at $1=e_{0}$ which is obviously true.
\ Assume for $k-1\geq0$ then the coefficient of $t^{k}$ on the left hand side
of the equation is
\[
\sum_{,j}m_{j,k}q^{j}+\sum_{,j}m_{j.k-1}q^{j}=\sum_{,j}m_{j,k}q^{j}%
+e_{k-1}(q,...,q^{n-1})
\]
by the inductive hypothesis. Also the coefficient of $t^{k}$ on the right hand
side of the equation $(\ast)$ is $e_{k}(1,q,...,q^{n-1})$. This completes the induction.
\end{proof}

\section{The upper bound}

On $\Omega_{n}$ put the inner product that is the tensor product of the usual
inner product on $\mathbb{R}[x_{1},...,x_{n}]$ and the inner product on
$\wedge\mathbb{R}^{n}$ such that the $dx_{i_{1}}\wedge dx_{i_{2}}\wedge
\cdots\wedge dx_{i_{k}}$ with $1\leq i_{1}<...<i_{k}\leq n$ form an
orthonormal basis. The next task is to study the orthogonal complement to the
ideal in $\Omega_{n}$ generated by the $S_{n}$ homogenous invariants of
positive total degree (here $\deg x_{i}=1$ and $\deg dx_{i}=1$). This implies
that we are studying the set of solutions to the following equations:%

\[
D_{k}\omega=0,\delta_{l}\omega=0,k=1,2,...,n,l=0,...,n-1.
\]
with
\[
D_{l}\sum_{1\leq i_{1}<...<i_{k}<n}h_{i_{1},...i_{k}}dx_{i_{1}}\wedge
dx_{i_{2}}\wedge\cdots\wedge dx_{i_{k}}=\sum_{i,1\leq i_{1}<...<i_{k}%
<n}\left(  \frac{\partial^{l}}{\partial x_{i}^{l}}h_{i_{1},...i_{k}}\right)
dx_{i_{1}}\wedge dx_{i_{2}}\wedge\cdots\wedge dx_{i_{k}}%
\]
and
\[
\delta_{l}\sum_{1\leq i_{1}<...<i_{k}<n}h_{i_{1},...i_{k}}dx_{i_{1}}\wedge
dx_{i_{2}}\wedge\cdots\wedge dx_{i_{k}}%
\]%
\[
=\sum_{,j,1\leq i_{1}<...<i_{k}<n}\left(  \frac{\partial^{l}}{\partial
x_{j}^{l}}h_{i_{1},...i_{k}}\right)  \iota(dx_{j})dx_{i_{1}}\wedge dx_{i_{2}%
}\wedge\cdots\wedge dx_{i_{k}}%
\]
and if $i_{1}<...<i_{k}$ then $\iota(dx_{j})dx_{i_{1}}\wedge\cdots\wedge
dx_{i_{k}}=0$ if $j\notin\{i_{1},...,i_{k}\}$ and if $j=i_{r}$ then
\[
\iota(dx_{j})dx_{i_{1}}\wedge\cdots\wedge dx_{i_{k}}=(-1)^{r-1}dx_{i_{1}%
}\wedge\cdots\wedge dx_{i_{r-1}}\wedge dx_{i_{r+1}}\wedge\cdots\wedge
dx_{i_{k}}.
\]
Thus we are studying the forms
\[
\omega=\sum h_{i_{1},...i_{k}}dx_{i_{1}}\wedge dx_{i_{2}}\wedge\cdots\wedge
dx_{i_{k}}%
\]
with $h_{i_{1},...i_{k}}$ an $S_{n}$ harmonic and $\delta_{k}\omega=0$ for
$k=0,1,...$. We will set $\delta_{k}=\delta_{k,n}$ when we need to take into
account the variables that are in play. Let $W_{r,s}$ be the space of all
$\omega=\sum h_{i_{1},...i_{s}}dx_{i_{1}}\wedge dx_{i_{2}}\wedge\cdots\wedge
dx_{i_{s}}$ with $h_{i_{1},...i_{s}}$ a harmonic of degree $r$.

Then the space we are considering consists of the elements $\omega\in W_{r,s}$
such that $\delta_{j}\omega=0$ for all $j=0,1,...,n-1$. We will denote this
space by $H_{r,s}$. \ Then $H_{r,s}$ is an $S_{n}$ invariant subspace of
$\ker\delta_{0}$ on $W_{r,s}$.

\begin{lemma}
\label{upper}Let $\xi_{r,s}$ be the character of the $S_{n}$ representation
$H_{r,s}$. If $\Lambda$ is a partition of $n$ then
\[
(\xi_{r,s},\chi_{\Lambda})\leq\sum_{\mu\vdash n}m_{r,\mu}(\chi_{\mu}%
\chi_{\left(  n-s,1^{s}\right)  },\chi_{\Lambda})\text{.}%
\]

\end{lemma}

\begin{proof}
$\ker\delta_{0}$ in $\wedge^{k}\mathbb{R}^{n}$ is $\wedge^{k}F$.
\end{proof}

\section{The Hilbert series of the alternants}

The purpose of this section is to derive the bigraded Hilbert series of the
alternants in the orthogonal complement to the ideal generated by the
invariants of positive degree in $\Omega_{n}$. The amazing, fact in this case,
is that the upper bound of the preceding section is a lower bound we will show
this by constructing a bihomogeneous basis. The ultimate formula for the
Hilbert series was suggested by Adriano Garsia. The upper bound described
above as a Hilbert series is

\begin{lemma}
An upper bound for
\[
\sum_{r,s}q^{r}t^{s}\sum_{\mu\vdash n}m_{r,\mu}(\chi_{\mu}\chi_{\left(
n-s,1^{s}\right)  },sgn)
\]
is%
\[
\prod_{j=1}^{n-1}(q^{j}+t).
\]

\end{lemma}

\begin{proof}
Note that
\[
(\chi_{\mu}\chi_{\left(  n-s,1^{s}\right)  },sgn)=(sgn\chi_{\mu}\chi_{\left(
n-s,1^{s}\right)  },1).
\]
This is non-zero if and only if $sgn\chi_{\mu}=\chi_{\left(  n-s,1^{s}\right)
}$ and then it is $1$. That is, if and only if $\mu=(s+1,1^{n-s-1})$.
Corollary \ref{mult} says that
\[
\sum_{j}q^{j}m_{j,\left(  k+1,1^{n-k-1}\right)  }=e_{n-k-1}(q,...,q^{n-1}).
\]
The lemma now follows.
\end{proof}

Set $\partial_{j}=\frac{\partial}{\partial x_{j}}$ and define the operator
$d_{j}$ for $j=0,1,2,...$by%
\[
d_{j}\sum_{1\leq i_{1}<i_{2}<...<i_{k}\leq n}f_{i_{1}i_{2}...i_{k}}dx_{i_{1}%
}\wedge dx_{i_{2}}\wedge\cdots\wedge dx_{i_{k}}=
\]%
\[
\sum_{1\leq i_{1}<i_{2}<...<i_{k}\leq n}^{n}\sum_{l}\partial_{l}^{j}%
f_{i_{1}i_{2}...i_{k}}dx_{l}\wedge dx_{i_{1}}\wedge dx_{i_{2}}\wedge
\cdots\wedge dx_{i_{k}}.
\]
Notice that $d_{j}:W_{r,s}\rightarrow W_{r-j,s+1}$.

\begin{lemma}
$d_{r}$ and $\delta_{s}$ satisfy the anticommutation relation%
\[
d_{r}\delta_{s}+\delta_{s}d_{r}=D_{s+r}%
\]

\end{lemma}

\begin{proof}
This is a direct consequence of
\[
\iota(dx_{j})dx_{k}\wedge dx_{i_{1}}\wedge dx_{i_{2}}\wedge\cdots\wedge
dx_{i_{k}}+dx_{k}\wedge\iota(dx_{j})dx_{i_{1}}\wedge dx_{i_{2}}\wedge
\cdots\wedge dx_{i_{k}}%
\]%
\[
=\delta_{jk}dx_{i_{1}}\wedge dx_{i_{2}}\wedge\cdots\wedge dx_{i_{k}}.
\]
If $n-1\geq m_{1}>m_{2}>...>m_{k}\geq1$ set
\[
\omega_{m_{1},m_{2},...,m_{k}}=d_{m_{1}}d_{m_{2}}\cdots d_{m_{k}}\Delta
\]
with
\[
\Delta=\prod_{1\leq i<j\leq n}(x_{i}-x_{j}).
\]
(the usual Vandermonde determinant). With the understanding that if $k=0$ then
$d_{m_{1}}d_{m_{2}}\cdots d_{m_{k}}\Delta=\Delta$ we have
\end{proof}

\begin{lemma}
$\omega_{m_{1},...,m_{k}}\in H_{\binom{n}{2}-\sum m_{i},k}$ and $s\omega
_{m_{1},...,m_{k}}=sgn(s)\omega_{m_{1},...,m_{k}}$.
\end{lemma}

\begin{proof}
By induction on $k$. If $k=0$ then it is clear that $\delta_{j}\Delta=0$ all
$j$. Assume for $k-1\geq0$ then
\[
\delta_{j}\omega_{m_{1},...,m_{k}}=\delta_{j}d_{m_{1}}\omega_{m_{2},...,m_{k}%
}=-d_{m_{1}}\delta_{j}\omega_{m_{2},...,m_{k}}+D_{j+m_{1}}\omega
_{m_{2},...,m_{k}}=0
\]
by the inductive hypothesis.
\end{proof}

\begin{lemma}
If $n-1\geq m_{1}>...>m_{k}\geq1$ then define $C(m_{1},...,m_{k})$ as follows:
let the complement of $\{n-m_{1},...,n-m_{k}\}$ be $\{j_{1},...,j_{n-k}\}$
with $j_{1}<j_{2}<...<j_{n-k}$ (note that $j_{n-k}=n$) \ then
\[
C(m_{1},...,m_{k})=(n-j_{1},...,n-j_{n-k-1}).
\]
One has%
\[
C(C(m_{1},...,m_{k}))=(m_{1},...,m_{k}).
\]
If
\[
\omega_{m_{1},...,m_{k}}=\sum_{1\leq i_{1}<...<i_{k}\leq n}h_{i_{1},...i_{k}%
}^{m_{1}...m_{k}}dx_{i_{1}}\wedge dx_{i_{2}}\wedge\cdots\wedge dx_{i_{k}}%
\]
then the coefficient of the monomial $x_{1}^{n-j_{1}}x_{2}^{n-j_{2}}\cdots
x_{n-k-1}^{n-j_{n-k-1}}$ in $h_{n-k+1,n-k,...,n}^{m_{1}...m_{k}}$ is $\pm
k!m_{1}!\cdots m_{k}!$. Finally if this monomial occurs in $\omega
_{r_{1},...,r_{k}}$ with $n-1\geq r_{1}>...>r_{k}\geq1$ then $r_{i}=m_{i}$ for
all $i=1,...,k$.
\end{lemma}

\begin{proof}
The assertion about $C\circ C$ is easily seen. Set $\rho=(n-1,n-2,...,1,0)$.
Note that up to non-zero multiple the coefficient of $dx_{i_{1}}\wedge
\cdots\wedge dx_{i_{k}}$ with $1\leq i_{1}<i_{2}<...<i_{k}\leq n$ is
\[
\sum_{t\in S_{k}}sgn(t)\partial_{i_{t1}}^{m_{1}}\cdots\partial_{i_{tk}}%
^{m_{k}}\Delta=\sum_{s\in S_{n},t\in S_{k}}sgn(s)sgn(t)\partial_{i_{t1}%
}^{m_{1}}\cdots\partial_{i_{ik}}^{m_{k}}x^{s\rho}.
\]
So to have the monomial $x_{1}^{n-j_{1}}x_{2}^{n-j_{2}}\cdots x_{n-k-1}%
^{n-j_{n-k-1}}$ appear with $j_{i}=n-k+i,i=1,...,k$ we must have
$sj_{1}=1,...,sj_{n-k-1}=n-k-1$ and $s(n-m_{1})=m-k+t1,...,s(n-m_{k})=n-k+tk$
for some $t\in S_{k}$ \ This uniquely determines $s$ so we can denote it by
$s_{t}$. Thus the coefficient of $x_{1}^{n-j_{1}}x_{2}^{n-j_{2}}\cdots
x_{n-k-1}^{n-j_{n-k-1}}$ is
\[
m_{1}!\cdots m_{k}!\sum_{t\in S_{k}}sgn(s_{t})sgn(t).
\]
Note that $sgn(s_{t})=sgn(t)sgn(s_{I})$ ($I$ is the identity permutation of
$1,...,k$). Thus the coefficient is $\pm k!m_{1}!\cdots m_{k}!$. \ If this
monomial appears with non-zero coefficients in $h_{n-k+1,n-k,...,n}%
^{r_{1}...r_{k}}$ with $n-1\geq r_{1}>...>r_{k}\geq1$ then since $s\rho$ is
just a rearrangement of $(n-1,n-2,...,1,0)$ we must have $C(r_{1}%
,...,r_{k})=C(m_{1},...,m_{k}).$This completes the proof.
\end{proof}

\begin{lemma}
The set
\[
\{\omega_{m_{1},...,m_{k}}|n-1\geq m_{1}>....>m_{k}\geq1,k=1,...,n-1\}\cup
\{\Delta\}
\]
is linearly independent.
\end{lemma}

\begin{proof}
It is enough to show that the $\omega_{m_{1},...,m_{k}}$ with $n-1\geq
m_{1}>....>m_{k}\geq1$ and fixed $k$ are linearly independent. The preceding
lemma implies that the functions $h_{i_{1},...i_{k}}^{m_{1}...m_{k}}$ (defined
in the statement) with $i_{j}=n-k+j,j=1,...,k$ are linearly independent.
\end{proof}

\begin{theorem}
The bigraded Hilbert series of the alternants in the orthogonal complement of
the ideal generated by the invariants of positive degree in $\Omega_{n}$ is
$\prod_{j=1}^{n-1}(q^{j}+t)$. That is if $\chi_{i,j}$ is the $S_{n}%
$--character of $H_{i,j}$ then
\[
\sum_{i,j}q^{i}t^{j}(\chi_{ij},sgn)=\prod_{j=1}^{n-1}(q^{j}+t).
\]

\end{theorem}

\begin{proof}
Lemma \ref{upper} says that the formula is an upper bound. The above lemma
says that bigraded Hilbert series of the space spanned by
\[
\{\omega_{m_{1},...,m_{k}}|n-1\geq m_{1}>....>m_{k}\geq1,k=1,...,n-1\}\cup
\{\Delta\}
\]
is
\[
q^{\binom{n}{2}}+\sum_{k=1}^{n-1}\sum_{n-1\geq m_{1}>....>m_{k}\geq1}%
q^{\binom{n}{2}-m_{1}-...-m_{k}}t^{k}%
\]%
\[
=q^{\binom{n}{2}}\left(  1+\sum_{k=1}^{n-1}\sum_{n-1\geq m_{1}>....>m_{k}%
\geq1}q^{-m_{1}-...-m_{k}}t^{k}\right)
\]%
\[
=q^{\binom{n}{2}}\prod_{j=1}^{n-1}(1+q^{-j}t)=\prod_{j=1}^{n-1}(q^{j}+t).
\]
Thus the upper bound is a lower bound!
\end{proof}

\section{The best possible upper bound on the degrees of the harmonic
$k$--forms}

In this section we will use coefficients in $\mathbb{C}$ rather than in
$\mathbb{R}$. We will prove a result that Marino Romero showed follows from
Zabrocki's conjecture. Let $\mathcal{I}$ denote the ideal generated by
$p_{1},...,p_{n},dp_{1},...,dp_{n}$ in $\Omega$. Here $p_{j}$ is the $j$-th
power sum. Note that
\[
d\mathcal{I}\subset\mathcal{I}.
\]
Indeed, if $\omega\in\mathcal{I}$ then
\[
\omega=\sum_{r}\sum_{1\leq i_{1}<...<i_{r}\leq n}f_{i_{1}i_{2}...i_{r}%
}dp_{i_{1}}\wedge\cdots\wedge dp_{i_{r}}%
\]
with $f_{i_{1}i_{2}...i_{r}}$ an element of the ideal $\mathcal{J}$ generated
by $p_{1},...,p_{n}$. Thus%
\[
\omega=\sum_{r}\sum_{1\leq i_{1}<...<i_{r}\leq n}df_{i_{1}i_{2}...i_{r}}\wedge
dp_{i_{1}}\wedge\cdots\wedge dp_{i_{r}}%
\]
But
\[
f_{i_{1}i_{2}...i_{r}}=\sum_{j}a_{i_{1}i_{2}...i_{r},j}p_{j}%
\]
so
\[
df_{i_{1}i_{2}...i_{r}}=\sum_{j}p_{j}da_{i_{1}i_{2}...i_{r},j}+\sum
_{j}a_{i_{1}i_{2}...i_{r},j}dp_{j}\in\mathcal{I}.
\]

Set $\Omega^{l,k}$ equal to the elements
\[
\sum_{1\leq i_{1}<...<i_{k}\leq n}f_{i_{1}i_{2}...i_{r}}dx_{i_{1}}\wedge
\cdots\wedge dx_{i_{k}}%
\]
with $f_{i_{1}i_{2}...i_{r}}$ a homogeneous polynomial of degree $l$.

\begin{theorem}
If $\omega\in\Omega^{l,k}$ and $l>\binom{n}{2}-\binom{k+1}{2}$ then $\omega
\in\mathcal{I}$. Furthermore the image of $\Omega^{\binom{n}{2}-\binom{k+1}%
{2},k}$ in $\Omega/\mathcal{I}$ is non-zero.
\end{theorem}

The proof of this result will take some preparation.

Let $h_{k}(x_{1},...,x_{n})$ be the degree $k$ complete homogeneous symmetric
polynomial in $n$ variables.

The proof of the following lemma is based on an argument of Neeraj Kumar that
he uploaded to Mathoverflow in 2012

\begin{lemma}
If $k\geq$ $2$ and $n\geq1$ then the polynomials $\partial_{i}h_{k},i=1,...,n$
have no common $0$ in $\mathbb{C}^{n}-\{0\}$.
\end{lemma}

\begin{proof}
It is obvious for $k=1$ and all $n\geq1$ and for $n=1$ and all $k\geq1$.

We prove the result by induction on $k$ and for each $k$ by induction on $n$.
Assume for $k-1$ we prove the result for $k$ by induction on $n$. Assume the
result for $n-1\geq1$. If $x\neq0$ but some $x_{j}=0$ and if $x$ is common $0$
of $\partial_{i}h_{k},i=1,...,n$ then $(x_{1},...,x_{j-1},x_{j+1},...,x_{n})$
is a common $0$ of
\[
\partial_{i}h_{k}(x_{1},...,x_{j-1},x_{j+1},...,x_{n}),i\neq j
\]
so the inductive hypothesis implies that $(x_{1},...,x_{j-1},x_{j+1}%
,...,x_{n})=0$. Thus we need only prove that if $x_{i}\neq0$ for all $i$ then
at least one of the $\partial_{i}h_{k}(x_{1},...,x_{n})$ is non-zero. So
assume this is false and we derive a contradiction.

We will use standard multi-index notation. That is $x^{J}=x_{1}^{j_{1}}\cdots
x_{n}^{j_{n}}$ for $J=(j_{1},...,j_{n})$ and all of the $j_{i}$are
non-negative integers. Also $|J|==j_{1}+...+j_{n}$ and $h_{k}=\sum_{\left\vert
J\right\vert =k}x^{J}$. In general%
\[
\partial_{i}h_{k}=\sum_{\left\vert J\right\vert =k-1}(j_{i}+1)x^{J}%
\]
so%
\[
\sum\partial_{i}h_{k}=\left(  k+n-1\right)  h_{k-1}.
\]
This implies

$(\ast)$ If $\partial_{i}h_{k}(x)=0$ for all $i$ then $h_{k-1}(x)=0.$

Also for each $i$ one has
\[
h_{k}(x_{1},...,x_{n})=\sum_{j=1}^{k}x_{i}^{j}h_{k-j}(x_{1},...,x_{i-1}%
,x_{i+1},...,x_{n})+h_{k}(x_{1},...,x_{i-1},x_{i+1},...,x_{n})
\]%
\[
=x_{i}\sum_{j=1}^{k}x_{i}^{j-1}h_{k-j}(x_{1},...,x_{i-1},x_{i+1}%
,...,x_{n})+h_{k}(x_{1},...,x_{i-1},x_{i+1},...,x_{n})
\]%
\[
=x_{i}h_{k-1}(x_{1},...,x_{n})+h_{k}(x_{1},...,x_{i-1},x_{i+1},...,x_{n}).
\]
Thus%
\[
\partial_{i}h_{k}(x_{1},...,x_{n})=h_{k-1}+x_{i}\partial_{i}h_{k-1}%
(x_{1},...,x_{n}).
\]
If $\partial_{i}h_{k}(x_{1},...,x_{n})=0$ for all $i$ then $(\ast)$ \ above
implies that $h_{k-1}(x_{1},...,x_{n})=0$. Hence it implies that
$x_{i}\partial_{i}h_{k-1}(x_{1},...,x_{n})=0$ for all $i$ and so if all
$x_{i}\neq0$ (which is what we are assuming) then $\partial_{i}h_{k-1}%
(x_{1},...,x_{n})=0$ which is a contradiction to the inductive hypothesis in
$k$.
\end{proof}

The following corollary is just a restatement of the above lemma.

\begin{corollary}
If $k\geq2$ and $n\geq1$ then the functions $\{\partial_{i}h_{k}|i=1,...,n\}$
form a system of parameters for $\mathbb{C}[x_{1},...,x_{n}]$.
\end{corollary}

Let $I$ denote the ideal in $\mathbb{C}[x_{1},...,x_{n}]$ generated by
$\partial_{i}h_{k},i=1,...,n$. This above corollary implies that
$\{\partial_{i}h_{k}|i=1,...,n\}$ is a regular sequence in $\mathbb{C}%
[x_{1},...,x_{n}]$ since $\mathbb{C}[x_{1},...,x_{n}]$ is a graded
Cohen-Macaulay algebra we have

\begin{corollary}
\label{key}The Hilbert series of $\mathbb{C}[x_{1},...,x_{n}]/I$ is
$(1+q+...+q^{k-2})^{n}$. In particular if $f\in\mathbb{C}[x_{1},...,x_{n}]$
and $f$ is homogeneous of degree $d>n(k-2)$ then $f\in I$.
\end{corollary}

\begin{proof}
Since $\{\partial_{i}h_{k}|i=1,...,n\}$ is a regular sequence in
$\mathbb{C}[x_{1},...,x_{n}]$, $\mathbb{C}[x_{1},...,x_{n}]$ is a finitely
generated free module on the algebra, $B$, generated by $1$ and $\{\partial
_{i}h_{k}|i=1,...,n\}$. Notice that $B$ is a polynomial algebra over
$\mathbb{C}$ generated by $n$ homogeneous elements each of degree $k-1$. Thus
the Hilbert series of $B$ is
\[
\frac{1}{(1-q^{k-1})^{n}}.
\]
This implies that the Hilbert series of $\mathbb{C}[x_{1},...,x_{n}]/I$ is%
\[
\left(  \frac{1-q^{k-1}}{1-q}\right)  ^{n}.
\]

\end{proof}

Recall that $\mathcal{J}_{n}$ is the ideal in $\mathbb{C}[x_{1},...,x_{n}]$
generated by the symmetric polynomials homogeneous of positive degree. It is
well known that
\[
h_{r}(x_{r},...,x_{n})\in\mathcal{J}_{n}.
\]
In fact, these elements form a Groebner basis for $\mathcal{J}_{n}$ for any
term order such that $x_{1}>...>x_{n}$.

\begin{lemma}%
\[
\partial_{i}h_{r}(x_{r},...,x_{n})dx_{r}\wedge\cdots\wedge dx_{n}%
\in\mathcal{I}_{n}%
\]
for $i=r,...,n$.
\end{lemma}

\begin{proof}
By the above
\[
\omega_{i}=h_{r}(x_{r},...,x_{n})dx_{r}\wedge\cdots\wedge dx_{i-1}\wedge
dx_{i+1}\wedge\cdots\wedge dx_{n}\in\mathcal{I}_{n}.
\]
So we have%
\[
d\omega_{i}=(-1)^{i-1}\partial_{i}h_{r}(x_{r},...,x_{n})dx_{r}\wedge
\cdots\wedge dx_{n}\in\mathcal{I}_{n}.
\]

\end{proof}

\begin{corollary}
\label{key2}If $f\in\mathbb{C}[x_{r},...,x_{n}]$ is homogeneous of degree
$l>(r-2)(n-r+1)$ then%
\[
fdx_{r}\wedge\cdots\wedge dx_{n}\in\mathcal{I}_{n}.
\]

\end{corollary}

\begin{proof}
Corollary \ref{key} implies that under the condition on $l,$ $f=\sum
u_{i}(x_{r},...,x_{n})\partial_{i}h_{r}(x_{r},...,x_{n})$. Thus%
\[
fdx_{r}\wedge\cdots\wedge dx_{n}=\sum u_{i}(x_{r},...,x_{n})\partial_{i}%
h_{r}(x_{r},...,x_{n})dx_{r}\wedge\cdots\wedge dx_{n}\in\mathcal{I}_{n}.
\]

\end{proof}

We are now ready to prove the Theorem. Assume that $\omega\in\Omega^{l,k}$ and
$l>\binom{n}{2}-\binom{k+1}{2}$ but $\omega\notin\mathcal{I}_{n}$. Then%
\[
\omega=\sum_{i_{1}<i_{2}<...<i_{k}}a_{i_{1}...i_{k}}(x)dx_{i_{1}}\wedge
\cdots\wedge dx_{i_{k}}%
\]
$a_{i_{1}...i_{k}}(x)$ homogeneous of degree $l$ and at least one of the terms
must satisfy
\[
a_{i_{1}...i_{k}}(x)dx_{i_{1}}\wedge\cdots\wedge dx_{i_{k}}\notin%
\mathcal{I}_{n}.
\]
We can apply an element $s\in S_{n}$ satisfying
\[
s\left(  dx_{i_{1}}\wedge\cdots\wedge dx_{i_{k}}\right)  =dx_{n-k+1}%
\wedge\cdots\wedge dx_{n}%
\]
to see that we have an element $a(x)dx_{n-k+1}\wedge\cdots\wedge dx_{n}%
\notin\mathcal{I}_{n}$ with $a(x)$ homogeneous of degree $l$. expanding
$a(x)$. Modulo the ideal $\mathcal{J}_{n}$, $a(x)$ can be expanded into
monomials with $x^{J}$ with $J=(j_{1},...,j_{n})$ and $0\leq j_{i}\leq i-1$.
Thus there exists such a monomial (of degree $l$) such that $x^{J}dx_{r}%
\wedge\cdots\wedge dx_{n}\notin\mathcal{I}_{n}$. We have
\[
\sum_{i<n-k+1}j_{i}\leq\sum_{i<n-k+1}(i-1)=\binom{n-k}{2}.
\]
Thus
\[
\sum_{i\geq n-k}j_{i}\geq l-\binom{n-k}{2}>\binom{n}{2}-\binom{n-k}{2}%
-\binom{k+1}{2}=k(n-k-1).
\]
This since Corollary \ref{key2} is stated for $r=n-k+1$(so $k=n-r+1$) this is
a contradiction.

\begin{proof}
This is best possible because if the operator $d_{j}$ are as in the previous
section then $d_{1}d_{2}\cdots d_{k}\Delta$ is harmonic non-zero and in
$\Omega^{l,k}$ with $l=\binom{n}{2}-\binom{k+1}{2}$.
\end{proof}

\begin{corollary}
If $\omega$ is a harmonic element of $\Omega_{n}$ then its total degree is at
most $\binom{n}{2}$. The space of harmonic elements of total degree $\binom
{n}{2}$ is two dimensional and has basis $\Delta,d\Delta$.
\end{corollary}

\section{A conjecture}

The previous results suggest the following (which is an analog of the operator
conjecture for diagonal harmonics).

\textbf{\ \ \ Conjecture} \ \ Let $H$ be the space of all $S_{n}$--harmonic
polynomials in one set of variables. Then the space of all $S_{n}$--harmonic
differential forms is%
\[
\sum_{k=0}^{n-1}\sum_{1\leq j_{1}<...<j_{k}\leq n-1}d_{j_{1}}d_{j_{2}}\cdots
d_{j_{k}}H.
\]

This is true for the alternants and so it implies, in particular, that the
harmonic differential forms are spanned by the partial derivatives of the alternants.

A version of this conjecture has been proposed by Zabrocki. In a context
closer to ours a result of this nature appears in [RW].

\begin{center}
{\Large References}

\bigskip
\end{center}

[RW] Brendon Rhoades and Andrew Timothy Wilson, Vandermondes in Superspace, To
appear.\smallskip

[S] Louis Solomon, Invariants of Finite Reflection Groups, Nagoya Math. J. 22
(1963), 57-64.

[W] Nolan R. Wallach, Invariant differential operators on reductive Lie
algebras and Weyl group representations, Jour. A.M.S, 6 (1998),
779-816.\smallskip

[Z] Mike Zabrocki, A module for the $\Delta$--conjecture, arXiv:1902.08966v1.

\end{document}